\documentclass[12pt]{amsart}
\usepackage[letterpaper,left=2.5cm,top=2.5cm,bottom=2.75cm, 
right=2.5cm]{geometry}

\usepackage{amssymb, amsmath, mathrsfs, amscd, esint, amsthm}
\usepackage{graphicx, xcolor, amsfonts, enumerate}

\newcommand{\R}{\mathbb{R}}

\newcommand{\N}{\mathbb{N}}

\newcommand{\eps}{\varepsilon}

\newcommand{\cO}{{\mathcal{O}}}

\newcommand{\xtilde}{{\widetilde{x}}}
\newcommand{\rtilde}{{\widetilde{r}}}

\newcommand{\Loneloc}{{L^1_{\text{loc}}}}
\newcommand{\BMO}{\textnormal{BMO}}
\newcommand{\BLO}{\textnormal{BLO}}
\newcommand{\VMO}{\textnormal{VMO}}
\newcommand{\Liploc}{\textnormal{Lip}_{\textnormal{loc}}}

\DeclareMathOperator*{\essinf}{ess\,inf}
\DeclareMathOperator{\supp}{supp}
\DeclareMathOperator{\rad}{rad}
\DeclareMathOperator{\dist}{dist}
\DeclareMathOperator{\loc}{loc}
\DeclareMathOperator{\glob}{glob}
\DeclareMathOperator{\diam}{diam}

\newtheorem{theorem}{Theorem}[section]
\newtheorem{lemma}[theorem]{Lemma}
\newtheorem{proposition}[theorem]{Proposition}
\newtheorem{corollary}[theorem]{Corollary}

\theoremstyle{definition}
\newtheorem{definition}[theorem]{Definition}
\newtheorem{example}[theorem]{Example}
\newtheorem{remark}[theorem]{Remark}

\numberwithin{equation}{section}

\begin{document}
	
	\title[Fractional Maximal Functions \& VMO]{Fractional maximal functions and mean oscillation on bounded doubling metric measure spaces}
	
	\author[Gibara, Kline]{Ryan Gibara, Josh Kline}
	\thanks{J. K. was partially supported by NSF
		grant \#DMS-2054960. 
		}
	
	\date{\today}
	
	
	\begin{abstract}
		Let $(X,d,\mu)$ be a doubling metric measure space.
		We consider the behaviour of the fractional maximal function $M^\alpha$ for $0\leq \alpha<Q$, where $Q$ is the doubling dimension, acting on functions of bounded mean oscillation ($\BMO$) and vanishing mean oscillation ($\VMO$). For $\alpha>0$, we additionally assume that the space is bounded. We show that $M^\alpha$ is bounded from $\BMO$ to $\BLO$, a subclass of $\BMO$, and maps $\VMO$ to itself when $\mu$ has the annular decay property. We also show by means of examples that the action of $M^\alpha$ is not continuous on these function spaces.
	\end{abstract}
	
	\maketitle
	
	\section{Introduction}
	
	For a locally integrable function $f$ on $\R^n$, its (uncentered) fractional Hardy-Littlewood maximal function is defined for $x\in \R^n$ and $0\leq \alpha<n$ as 
	$$
	M^\alpha f(x)=\sup_{B\ni x}r^\alpha\fint_{B}\!|f|=\sup_{B\ni x}\frac{r^\alpha}{|B|}\int_{B}\!|f|,
	$$
	where the supremum is taken over all balls $B\subset\R^n$ of radius $r>0$ containing $x$. When $\alpha=0$, we write $M^0=M$, which is the usual (uncentered) Hardy-Littlewood maximal function. Classical objects in analysis, maximal functions have connections to differentiation of the integral, singular integrals, and potential theory. A common variant of $M^\alpha$ is when the balls are replaced with cubes (which, in this paper, will always mean with sides parallel to the coordinate axes). 
	
	An active line of research involves studying the behaviour of maximal functions as they act on various types of functions. Set $p^*=np/(n-\alpha p)$, which is strictly larger than $p$ when $\alpha>0$. It follows from comparison with the Riesz potential that if $0\leq \alpha<n/p$, then $M^\alpha$ is bounded from $L^p(\R^n)$ to $L^{p^*}(\R^n)$ if $1<p<\infty$ and from $L^p(\R^n)$ to $L^{p^*,\infty}(\R^n)$ if $p=1$, see \cite[Chapter V, Theorem 1]{st}. When $\alpha=0$, we recover the well-known Hardy-Littlewood-Wiener theorem. It follows from the fact that $M^\alpha$ commutes with translations that for $0\leq\alpha<n/p$ it is bounded from $W^{1,p}(\R^n)$ to $W^{1,p^*}(\R^n)$, see \cite{kin,ksa} (although these results are shown for the centered analogue of the maximal function, they also holds for $M^\alpha$).

	In this work, we are interested in functions that are neither in $L^p$ nor in $W^{1,p}$, but rather for which we have some information about their mean oscillations over a specified collection of sets, usually cubes or balls. The quintessential such space, first defined in \cite{jn}, is the space of functions of bounded mean oscillation, $\BMO(\R^n)$, the collection of all locally integrable functions $f$ on $\R^n$ for which 
	$$
	\|f\|_{\BMO}=\sup_{Q}\fint_{Q}\!|f-f_Q|<\infty,
	$$
	where $f_Q$ denotes the mean of $f$ over $Q$ and the supremum is taken over all cubes $Q\subset\R^n$. As $\BMO$ functions need not be integrable, it is not necessary that their image under a maximal function be finite at any point; as such, we exclude those for which the maximal function is identically infinite. In \cite{bdvs}, it was shown that $M$, taken with respect to cubes, is bounded on $\BMO(\R^n)$. This was then improved in \cite{ben}, where it was shown that the image of $\BMO(\R^n)$ is contained inside a subclass, $\BLO(\R^n)$, see Definition \ref{def:BLO}. The same proofs also hold for $M$ taken with respect to balls, see \cite{dgy}. 
	
	Introduced in \cite{sa}, the space of functions of vanishing mean oscillation, $\VMO(\R^n)$, is the subspace of $f\in \BMO(\R^n)$ such that
	$$
	\lim_{r\rightarrow 0^+}\sup_{|Q|\leq r}\fint_{Q}\!|f-f_Q|=0,
	$$
	where the supremum is taken over cubes $Q$ of measure at most $r$ (note that in this definition, we could just have well considered cubes of side length at most $r$). Equivalently, Sarason's theorem tells us that $f$ is in $\VMO(\R^n)$ if and only if it can be approximated in the $\BMO(\R^n)$ norm by uniformly continuous $\BMO(\R^n)$ functions. This vanishing mean oscillation condition is a common minimal regularity condition on the coefficients of partial differential equations. It was only recently shown in \cite{shahab} that $M$, taken with respect to cubes, maps $\VMO(\R^n)$ to itself.
	
	For functions only defined on a fixed cube $Q_0\subset\R^n$, the space $\BMO(Q_0)$, as well as its subclass $\BLO(Q_0)$ and its subspace $\VMO(Q_0)$, are defined by restricting to mean oscillations of subcubes $Q\subset Q_0$. Similarly to the global case, it is known that $M$, taken with respect to cubes, is bounded from $\BMO(Q_0)$ to $\BLO(Q_0)$, see \cite{ben}. 
	
	
	The goal of the present paper is to extend the known $\BMO(Q_0)$ and $\VMO(Q_0)$ results to the fractional maximal function. We do so in the even more general setting of a bounded doubling metric measure space. In recent decades, the analysis on metric measure spaces has become a field of active study, showing that many of the fundamental aspects of analysis and first-order calculus can be adequately and fruitfully generalized to spaces with no linear or differentiable structure. When the measure is doubling, see Definition \ref{def:doubling}, much of classical harmonic analysis can be extended to this setting. In particular, an analogue of the Hardy-Littlewood-Weiner theorem, and therefore Lebesgue's differentiation theorem, holds, see \cite[Theorems 1.8 and 2.2]{H}.
	
	In this general setting of a metric measure space $(X,d,\mu)$, the definitions of $\BMO(X)$ and $\BLO(X)$ can be naturally generalized by replacing cubes with balls $B\subset X$. When considering the subspace $\VMO(X)$, however, there is a choice to be made. One could generalize $\VMO(X)$ by looking at balls of vanishing measure; however, unlike, the Euclidean setting, this is not equivalent to considering balls of vanishing radius. We choose the definition in terms of vanishing radii due to its more obvious connection to the fractional maximal function and to continuity properties of functions. Section \ref{sec:VMO} is dedicated to showing some basic properties of $\VMO$ thus defined when $\mu$ is a doubling measure; in fact, we are able to show in Theorem \ref{thm:VMOChar} that with this definition of $\VMO(X)$, an analogue of Sarason's theorem holds. 
	
	The main bulk of the paper is spent in Section \ref{sec:fs} investigating the action of $M^\alpha$ on function spaces. For the $\BMO$ result, see Theorem \ref{theorem:BMO}, we require that the measure be doubling and, if $\alpha>0$, the space be bounded; on the other hand, the $\VMO$ result, see Theorem \ref{theorem:VMO}, requires the additional assumption of an annular decay property on the measure, see Definition \ref{def:ann-decay}. This property is known to be both necessary and sufficient for maximal functions to map Lipschitz functions to Lipschitz functions, see \cite{buck}. The main results of the paper are summarized in the following.
	
	\begin{theorem}
		Let $(X,d,\mu)$ be a doubling metric measure space and $\alpha\geq 0$. When $\alpha>0$, we assume that $(X,d,\mu)$ is bounded with lower mass bound exponent $Q$. If $0\leq\alpha< Q$, then $M^\alpha$ is bounded from $\BMO(X)$ to $\BLO(X)$. Furthermore, if $\mu$ satisfies an annular decay condition, then $M^\alpha$ maps $\VMO(X)$ to itself. 
	\end{theorem}
	
	As a Euclidean consequence of this theorem, we obtain the following, see Corollary \ref{cor:Euclid}.
	
	\begin{corollary}
		Let $0\leq \alpha<n$ and $f\in\BMO(Q_0)$ for some fixed cube $Q_0\subset \R^n$. Then, $M^\alpha f \in \BLO(Q_0)$ with $\|M^\alpha f\|_{\BLO}\leq C\|f\|_{\BMO}$, where $C$ depends on $n$ and $\alpha$. Moreover, if $f\in\VMO(Q_0)$, then $M^\alpha f\in\VMO(Q_0)$.
	\end{corollary}
	
	This theorem complements the known $\alpha=0$ results in the literature by extending the $\BMO-\BLO$ result to the fractional maximal function. For the $\VMO-\VMO$ result, however, this is new even for $\alpha=0$, complementing the result of \cite{shahab} for $\VMO(\R^n)$.
	
	In general, for non-linear operators, boundedness does not imply continuity. The most famous example of a non-linear operator that is bounded but not continuous is the symmetric decreasing rearrangement, which is bounded on $W^{1,p}(\R^n)$ for all $1\leq p <\infty$ and any dimension, but is discontinuous unless $n=1$, see \cite{almg,cor}. The same operator is bounded on $\BMO(\R^n)$, see \cite{bdg}, but is not continuous, even if $n=1$, see \cite{BDG2}. As the Hardy-Littlewood maximal function is pointwise sublinear, the continuity of $M^\alpha$ on $L^p(\R^n)$ follows immediately from its boundedness. In Sobolev spaces, it was shown in \cite{lui} that $M$ is continuous from $W^{1,p}(\R^n)$ to $W^{1,p^*}(\R^n)$. 
	The question of the continuity of $M$ on $\BMO(\R^n)$ was recently addressed in \cite{shahab}, where a counterexample was constructed. In Section \ref{sec:cont}, we provide a simpler example when $n=1$, and also show that $M^\alpha$ is discontinuous on $\BMO(X)$ and $\VMO(X)$ when $X$ is an interval in $\R$. 
	
	As an aside, we mention that there is extensive ongoing research concerning the behaviour of maximal functions on Sobolev and related spaces when $p=1$. It is clear that $M$ cannot map $W^{1,1}(\R^n)$ to itself, as the image of a non-zero $L^1(\R^n)$ function under $M$ is necessarily non-integrable. The so-called $W^{1,1}$ conjecture, first posited in \cite{ho}, asks whether the map $f\mapsto |\nabla Mf|$ is bounded from $W^{1,1}(\R^n)$ to $L^1(\R^n)$. This conjecture was answered in the affirmative by \cite{tan} in dimension 1 and recently by \cite{weight} in arbitrary dimensions when considering $M$ with respect to cubes. The literature on the regularity of maximal functions in this endpoint case is vast; see, for instance, \cite{alp, bel, cmp, gk, mad, weight2}.

	\section{Preliminaries} 
	
	In this paper, $(X,d,\mu)$ will always denote a \emph{metric measure space}. That is, $(X,d)$ is a metric space equipped with a Borel measure $\mu$ that satisfies $0<\mu(B)<\infty$ for all balls $B\subset X$ of positive and finite radius. It is immediate that such a measure is $\sigma$-finite, and it also follows that $(X,d)$ is separable, see \cite[Lemma~3.3.30]{HKST} for example. We will assume throughout the paper that balls come with a prescribed center and radius to avoid the ambiguity coming from the fact that they are, in general, not unique for a given ball. The notation $\rad(B)$ will sometimes be used to denote this specified radius. For a constant $C>0$, the notation $CB$ will be used to represent the ball with the same center as $B$ but whose radius equals $C\rad(B)$. We will also assume that $X$ contains more than one point.
	
	Throughout this paper, we let $C$ denote a constant which depends, unless otherwise specified, on the doubling constant (see below). The exact value of $C$ is not of interest here and may vary at each occurrence, even within the same line.  For quantities $A$ and $B$, we use the notation $A\lesssim B$ to mean that there exists such a constant $C>0$ such that $A\le CB$. 
	
	
	\subsection{Doubling measures}
	
	In this subsection, we will define some important classes of measures that will be needed in later results. Prime among these measures are those that exhibit a doubling condition. 
	
	\begin{definition}\label{def:doubling}
		We say that $\mu$ is a \emph{doubling measure} if there exists a constant $C_\mu\geq 1$, called the \emph{doubling constant} of $\mu$, such that 
		$$
		\mu(B(x,2r))\leq C_\mu\,\mu(B(x,r))
		$$
		for all $x\in X$ and $0<r<\infty$. We say that $(X,d,\mu)$ is a \emph{doubling metric measure space} when $(X,d,\mu)$ is a metric measure space and $\mu$ is a doubling measure. 
	\end{definition}
	
	Doubling measures play an important role in the analysis of metric measure spaces. In particular, if a metric space is equipped with a doubling measure, then as a consequence, the Lebesgue differentiation theorem holds \cite[Chapter~1]{H}.  In such metric spaces, one also obtains standard $L^p$ boundedness of the Hardy-Littlewood maximal function and the John-Nirenberg inequality, see for example \cite[Chapter~3]{bb}.
	
	\begin{remark}\label{remark:doubling-constant}
		As we are assuming that $X$ contains more than one point, we have that $C_\mu\geq 2$, see \cite{sd}.
	\end{remark}
	
	\begin{lemma}\cite[Lemma~3.3]{bb}
		A measure $\mu$ is doubling if and only if there exist an exponent $Q\geq 0$ and a constant $0<c\leq 1$ such that
		\begin{equation}\label{RLMB}
			\frac{\mu(B(x,r))}{\mu(B(y,R))}\geq c\left( \frac{r}{R} \right)^Q
		\end{equation}
		for all $y\in X$, $x\in B(y,R)$, and $0<r\leq R<\infty$. 
	\end{lemma}
	
	A measure satisfying \eqref{RLMB} is said to satisfy a \emph{relative lower mass bound} and $Q$ is known as the \emph{relative lower mass bound exponent} of the measure. The exponent $Q$ plays the role of dimension for a doubling metric measure space. 
	
	
	
	\begin{remark}\label{remark:lower-mass-bound-exponent}
		As $r\leq R$ in \eqref{RLMB}, increasing the value of $Q$ makes the right-hand side of the inequality smaller. As such, we may assume without loss of generality that $Q>0$. 
	\end{remark}
	
	
	Another important type of measure are those for which a global analogue of the previous definition holds. 
	
	\begin{definition}
		A measure $\mu$ is said to satisfy a {\emph {lower mass bound}} with {\emph {lower mass bound exponent}} $Q>0$ if there exists some constant $0<c_L\leq 1$ such that 
		\begin{equation}\label{LMB}
			\mu(B(x,r))\geq c_L\,r^Q
		\end{equation}
		for all $x\in X$ and $0<r< 2\diam(X)$. 
	\end{definition}
	
	Note that if $(X,d,\mu)$ is a bounded doubling metric measure space, then by taking $R$ sufficiently large, say $R=2\diam(X)$, the doubling condition on $\mu$ implies that $\mu(X)<\infty$. Furthermore, the relative lower mass bound can then be upgraded to a lower mass bound with the same exponent, but with $c_L$ depending on $\diam(X)$, $\mu(X)$, and $Q$. Doubling measures on bounded spaces are not the only measures satisfying \eqref{LMB}, however. We say that a measure is {\it Ahflors $Q$-regular} for some $Q>0$, if there exists a $C\geq 1$ such that $C^{-1}r^Q\leq\mu(B(x,r))\leq Cr^Q$ holds for all $x\in X$ and $0<r<2\diam(X)$. Such measures need not be finite, but satisfy a combination of the lower mass bound and doubling. 
	
	\begin{remark}\label{remark:finite-measure-equals-finite-diameter}
		In the above paragraph, it is pointed out that a doubling metric measure space being bounded implies that it is of finite total measure. In fact, the converse is true as well, see \cite[Lemma 1.9]{bc}.  
	\end{remark}
	
	\begin{definition}\label{def:ann-decay}
		A measure $\mu$ is said to satisfy an \emph{annular decay property} if there exist an exponent $0<\beta\leq 1$ and a constant $C_{\beta}\geq 1$ such that 
		\begin{equation}\label{ADP}
			\mu(B(x,R)\setminus B(x,r))\leq C_{\beta}\,\left( \frac{R-r}{R} \right)^\beta\mu(B(x,R))
		\end{equation}
		for all $x\in X$ and $0<r\leq R<\infty$.
	\end{definition}
	
	
	
	First introduced in \cite[p. 125]{cm} in the setting of manifolds and in \cite{buck} in the setting of metric spaces, this concept, or a variant thereof, has been studied in connection to fractional maximal functions \cite{hlnt}, 
	Hardy inequalities and $Tb$ theorems \cite{djs,rou}, Muckenhoupt weights and reverse H\"older classes \cite{ks,kurk}, and potential theory and capacitary estimates \cite{agg, bbl}. The present paper investigates its connection to the behaviour of maximal functions on certain spaces defined by mean oscillation.
	
	\begin{remark}
		Although general doubling measures may not satisfy an annular decay property, if $\mu$ is doubling and $(X,d)$ is a length space, then $\mu$ does satisfy an annular decay property \cite[Corollary 2.2]{buck}. 
	\end{remark}
	
	
	
	
	\subsection{Discrete convolutions}
	
	In the context of doubling metric measure spaces, we are able to define the so-called \emph{discrete convolution} of a function $f\in L^1_{\text{loc}}(X)$. A metric-measure analogue of the usual smooth convolution in Euclidean harmonic analysis, the discrete convolution approximates the local behavior of $f$ by its averages on balls of a fixed scale. 
	
	The first step in the construction of the discrete convolution of $f$ is the following covering lemma. See for example \cite[Appendix~B.7]{g} and \cite[Chapter 3.3]{HKST}.
	
	\begin{lemma}\label{lemma:cover}
		Let $\mu$ be doubling. For any $\delta>0$, there exists a countable cover $\{B_i\}$ of $X$ by balls with $\rad(B_i)=\delta$ such that the collection $\{\frac{1}{5}B_i\}$ is disjoint. Moreover, for all $K>0$, the collection $\{KB_i\}$ has bounded overlap in the sense that there exists a constant $C\ge 1$, depending on $K$ and $C_\mu$, such that 
		$$
		\sum_{i}\chi_{KB_i}\le C\,.
		$$
	\end{lemma}
	
	With this cover in hand, we are able to build a \emph{Lipschitz partition of unity} subordinate to it as follows.  Again, see \cite[Appendix~B.7]{g}, for example. 
	
	\begin{lemma}\label{lemma:partition}
		Let $\mu$ be doubling, $\delta>0$, and $\{B_i\}$ be a countable cover of $X$ given by Lemma \ref{lemma:cover}. There exists a constant $C\geq 1$ such that, for each $i$, there exists a $C/\delta$-Lipschitz function $\varphi_i$ satisfying $0\le\varphi_i\le 1$, $\supp(\varphi_i)\subset2B_i,$ and $\sum_i\varphi_i\equiv 1$.  Here, the constant $C$ depends only on $C_\mu.$
	\end{lemma}
	
	Now we have the tools necessary to define the discrete convolution of $f\in L^1_{\text{loc}}(X)$ at scale $\delta>0$ when $\mu$ is doubling. Let $\{B_i\}$ be the countable cover given by Lemma \ref{lemma:cover} and let $\{\varphi_i\}$ be the partition of unity subordinate to this cover given by Lemma \ref{lemma:partition}. We define the discrete convolution of $f$ to be the function
	\begin{equation}\label{discrete}
		f_\delta:=\sum_i f_{B_i}\varphi_i\,,
	\end{equation}
	where we employ the notation $f_{A} := \fint_{A} f\,d\mu$ for any $\mu$-measurable set $A\subset X$ of positive and finite measure. The sum defining $f_\delta$ is a locally finite sum by the bounded overlap of the cover, and so it follows that $f_\delta\in\Liploc(X)$, the space of locally Lipschitz functions on $X$.  
	
	\subsection{Mean oscillation}
	
	In this subsection, we define the most basic of all function spaces defined by mean oscillation, $\BMO$. First introduced in \cite{jn} for functions on a cube in $\R^n$, $\BMO$ is the space of functions of bounded mean oscillation. This space has played a crucial role in the areas of harmonic analysis and partial differential equations where $L^\infty$ is often too pathological, necessitating a replacement to act as the endpoint for the $L^p$-scale as $p\rightarrow\infty$. 
	
	Although the original definition in \cite{jn} involved measuring the mean oscillation of locally integrable functions on subcubes of a fixed cube in $\R^n$, one can also consider the space globally on all of $\R^n$. Furthermore, one could have chosen to measure mean oscillation on balls rather than cubes, resulting in a space that is isomorphic to the one defined on cubes. This leads to a definition of $\BMO$ that is more natural in the context of a metric space, and can be traced back at least as far as \cite{cw}. 
	
	\begin{definition}
		Let $f \in L^1_{\text{loc}}(X)$ and $1\leq p <\infty$. The {\em $p$-mean oscillation} of $f$ on a ball $B\subset X$ is defined to be 
		$$
		\cO_p(f,B):=\left(\fint_{B}\!|f-f_B|^p\,d\mu\right)^{1/p}\,,
		$$
		where $f_B := \fint_B f\,d\mu$. We write $\cO:=\cO_1$.
	\end{definition}
	
	For a fixed ball $B$, the $p$-mean oscillation of $f$ does not change if a constant is added to it. Moreover, $\cO_p(f,B)=0$ if and only if $f$ is almost everywhere equal to some constant on $B$.
	
	\begin{definition}
		We say that $f\in L^1_{\text{loc}}(X)$ has {\em bounded $p$-mean oscillation} for $1\leq p <\infty$, written $f \in \BMO^p(X)$, if 
		$$
		\|f\|_{\BMO^p(X)}:=\sup_{B\subset X} \cO_p(f,B)<\infty\, .
		$$
		We write $\BMO(X):=\BMO^1(X)$.
	\end{definition}
	As with the $p$-mean oscillation itself, adding a constant to a function $f$ does not change $\|f\|_{\BMO^p(X)}$. Also, $\|f\|_{\BMO^p(X)}=0$ if and only if $f$ is almost everywhere equal to some constant. As such, this is technically only a semi-norm, although we will often nonetheless refer to it as the $\BMO$-norm. 
	
	We begin with a few facts about $\BMO^p$ and its norm that will be used in subsequent sections. By the triangle and Jensen's inequalities,
	\begin{equation}\label{mean-constant}
		\fint_{B}\!|f-f_B|^p\,d\mu\leq 2^p\,\fint_{B}\!|f-K|^p\,d\mu
	\end{equation}
	for any constant $K$. Choosing $K=0$, it is easy to see that if $f\in L^\infty(X)$, then $f\in\BMO^p(X)$ with $\|f\|_{\BMO^p(X)}\leq 2 \|f\|_{L^\infty(X)}$. Also by Jensen's inequality, for any ball $B\subset X$,
	$$
	\frac{1}{2^p}\fint_{B}\fint_{B}\!|f(x)-f(y)|^p\,d\mu(x)d\mu(y)\leq \fint_{B}\!|f-f_B|^p\,d\mu\leq \fint_{B}\fint_{B}\!|f(x)-f(y)|^p\,d\mu(x)d\mu(y).
	$$
	
	The following lemma allows us to compare the difference in the mean of a $\BMO$ function on nested balls. In the setting of doubling metric measure spaces, this can be found in \cite[Lemma 15]{vg}, see also \cite[Lemma 7.1]{HKNT}.
	
	\begin{lemma}\label{lem:jones}
		Let $\mu$ be doubling and $f\in\BMO(X)$. There exists a constant $C\geq 1$ such that for any two balls $B_0:=B(x_0,r_0)\subset B(x_1,r_1)=:B_1$,
		$$
		|f_{B_1}-f_{B_0}|\leq C\, \log\left(\frac{C\,r_1}{r_0}\right)\,\|f\|_{\BMO(X)}.
		$$
	\end{lemma}
	
	\begin{remark}\label{remark:john-nirenberg}
		When $\mu$ is doubling, a John-Nirenberg inequality holds for $\BMO(X)$, which implies that $\BMO^p(X)\cong\BMO(X)$ for all $1<p<\infty$ \cite[Chapter~3.3, 3.4]{bb}.
	\end{remark}
	
	We also need to define the following, which can be seen to be a subclass of $\BMO(X)$ by taking $K=\essinf\limits_{B}f$ in \eqref{mean-constant}. 
	
	\begin{definition}\label{def:BLO}
		We say that $f\in L^1_{\text{loc}}(X)$ has {\em bounded lower oscillation}, written $f \in \BLO(X)$, if 
		$$
		\|f\|_{\BLO(X)}:=\sup_{B\subset X} \fint_{B}\![f-\essinf_{B}f]\,d\mu<\infty\,.
		$$
	\end{definition}
	
	
	\section{Vanishing Mean Oscillation}\label{sec:VMO}
	An important subspace of $\BMO(\R^n)$ is the space of functions of vanishing mean oscillation, $\VMO(\R^n)$. This space was first defined by Sarason in \cite{sa}, where many of its properties were proven; in particular, $\VMO$ coincides with the closure of the uniformly continuous functions in the $\BMO$-norm. To the best knowledge of the authors, $\VMO$ has very seldom been considered in the setting of metric measure spaces. For the rest of this section, we assume that $(X,d,\mu)$ is a doubling metric measure space, and we show some basic properties of $\VMO(X)$, including an analogue of the aforementioned result of Sarason. 
	
	\begin{definition}
		We define the {\em modulus of oscillation} of $f\in L^1_{\text{loc}}(X)$ to be 
		$$
		\omega_p(f,r):=\sup_{\rad(B)\leq r}\cO_p(f,B)\,,
		$$
		where the supremum is taken over all balls $B\subset X$ of radius at most r.
	\end{definition}
	
	
	\begin{definition}
		We say that $f\in\BMO^p(X)$ is in $\VMO^p(X)$ if
		$$
		\omega_p(f,0):=\lim_{r\rightarrow{0}^+}\omega_p(f,r)=0\,.
		$$
		We write $\VMO(X):=\VMO^1(X)$.
	\end{definition}
	
	
	Note that we are making a choice to define $\VMO^p(X)$ in terms of balls with vanishing radii. One could just as easily define an analogous space where the vanishing mean oscillation occurs along balls with vanishing measure, which we denote by $\VMO_\mu^p(X)$. The spaces $\VMO^p(X)$ and $\VMO_\mu^p(X)$ are, {\it a priori}, different, as shown for $p=1$ by the following examples.
	
	
	\begin{example}
		Consider $X=[0,\infty)\subset\R$ with the Euclidean metric, and define the weight $\omega(x)=\sum_{n=0}^\infty\frac{1}{\sqrt{n+1}}\chi_{I_n}$, where $I_n=[n,n+1)$ for any non-negative integer $n$. This is a doubling weight, and we endow $X$ with the measure $d\mu(x)=\omega(x)\,dx$, where $dx$ denotes the Lebesgue measure on $X$. Denote by $\mathcal{O}_\omega$ the mean oscillation taken with respect to the measure $\mu$ and by $\mathcal{O}$ the mean oscillation with respect to Lebesgue measure.
		
		For 
		\[
		g(x)=
		\begin{cases}
			2x,&0\le x<1/2\\
			-2x+2,&1/2\le x<1,
		\end{cases}
		\]
		set $f_n(x)=g(x-n)$ for each non-negative integer $n$, and then define $f(x)=\sum_{n=0}^\infty f_n(x)$. Note that the each of the functions $f_n$ is non-zero only on the interval $I_n$, on which $\omega$ is constant, and the intervals $\{I_n\}$ are mutually disjoint. Hence, by properties of the mean oscillation, we have that for each $n$,
		$$
		\mathcal{O}_\omega(f,I_n)=\mathcal{O}_\omega(f_n,I_n)=\mathcal{O}(f_n,I_n)=\mathcal{O}(g,I_0)=\frac{1}{4}.
		$$
		From this is follows that $f\not\in\VMO_\mu(X)$.  Indeed, since $\mu(I_n)=\frac{1}{\sqrt{n+1}}$ is decreasing in $n$, for each $0<r<1$ we have that
		$$
		\sup_{\mu(I)\leq r}\mathcal{O}_\omega(f,I)\geq \sup_{n\geq r^{-2}}\mathcal{O}_\omega(f,I_n)=\frac{1}{4}>0.
		$$
		
		
		By construction, $f$ is a bounded 2-Lipschitz function, and so it follows that $f\in\VMO(X)$. Indeed, being bounded implies that $f\in\BMO(X)$ and the Lipschitz condition implies that for any interval $I$ of radius at most $r>0$, we have
		$$
		\mathcal{O}_\omega(f,I)\leq \fint_{I}\fint_{I}\!|f(x)-f(y)|\,d\mu(x)d\mu(y)\leq 2\ell(I)\leq 4r. 
		$$
		Therefore $f\in\VMO(X)\setminus\VMO_\mu(X).$
	\end{example}
	
	\begin{example}
		Consider the same setup as in the previous example, but this time with the doubling weight $\omega(x)=\sum_{n=0}^{\infty}(n+1)\chi_{I_n}$ and $f_n(x)=g((n+1)(x-n^2))$. Define $f(x)=\sum_{n=1}^
		\infty f_n(x)$. Each of the functions $f_n$ is non-zero only on the interval $I'_{n^2}=[n^2,n^2+\frac{1}{n+1}
		)$, $n\geq 1$, and the intervals $\{I'_{n^2}\}$ are mutually disjoint.  Each $I'_{n^2}$ is a subinterval of $I_{n^2}$, on which again $\omega$ is constant.   
		As before, we have by properties of the mean oscillation that for each $n$,
		$$
		\mathcal{O}_\omega(f,I'_{n^2})=\mathcal{O}_\omega(f_n,I'_{n^2})=\mathcal{O}(f_n,I'_{n^2})=\mathcal{O}(g,(0,1))=\frac{1}{4}.
		$$
		From this is follows that $f\not\in\VMO(X)$.  Indeed, since $\rad(I'_{n^2})=\frac{1}{2(n+1)}$ is decreasing in $n$, for each $0<r<1/2$ we have that
		$$
		\sup_{\rad(I)\leq r}\mathcal{O}_\omega(f,I)\geq \sup_{n\geq (2r)^{-1}}\mathcal{O}_\omega(f,I'_{n^2})=\frac{1}{4}>0.
		$$
		
		
		For $0<r<1/2$, consider an interval $I$ such that $0<\mu(I)<r$. 
		We claim that $I$ intersects at most one of $\{I'_{n^2}\}_{n\geq 1}$. To show this, note first that $I$ can intersect at most two of the $\{I_n\}_{n\geq 1}$; otherwise, $I$ contains some $I_m$ for $m\geq 1$, and so $r\geq\mu(I)\geq \mu(I\cap I_m)=(m+1)>1/2$. Hence, if $I$ intersects $I_m$ for $m\geq 1$, then either $I\subset I_m$ or it intersects exactly one of $I_{m-1}$ or $I_{m+1}$. In any of these cases, we have that $r\geq \mu(I)\geq m\ell(I)$, and so $\ell(I)\leq r/m<1/2$. In particular, if $I$ intersects $I'_{k^2}$, then it intersects $I_{k^2}$, and so $\ell(I)\leq r/k^2<1/2$. From the separation of the $\{I'_{n^2}\}$, the claim follows.
		
		Fix $0<r<1/2$ and an interval $I$ such that $0<\mu(I)<r$. Now, assume that $I$ intersects $I'_{k^2}$ for some $k\geq 1$. From the above, $\ell(I)\leq r/k^2$ and $I$ cannot intersect any other of the $\{I'_{n^2}\}$. Thus, as $f$ is $2(k+1)$-Lipschitz on $I'_{k^2}$, it is also $2(k+1)$-Lipschitz on $I$, and so we have that 
		\begin{align*}
			\cO_\omega(f,I)&\le\fint_I\fint_I|f(x)-f(y)|d\mu(y)d\mu(x)\le 2(k+1)\ell(I)<\frac{2(k+1)r}{k^2}\le 4r.
		\end{align*}
		If $I$ does not intersect any of the $\{I'_{n^2}\}$, then $\cO_\omega(f,I)=0,$ since $f$ is non-zero only on the intervals $\{I'_{n^2}\}$.  In either case, we have that $\cO_\omega(f,I)\le 4 r$ for all intervals $I$ such that $\mu(I)<r$, and so $f\in\VMO_\mu(X)$. Therefore, $f\in\VMO_\mu(X)\setminus\VMO(X).$
	\end{example}
	
	On the other hand, when $\mu$ satisfies $r^{Q_2}\lesssim\mu(B(x,r))$ for all $x\in X$ and $0<r<2\diam(X)$, and some $Q_2>0$, then $\VMO^p(X)\subset\VMO_\mu^p(X)$; analogously, when $\mu$ satisfies $\mu(B(x,r))\lesssim r^{Q_1}$ for all $x\in X$ and $0<r<2\diam(X)$, and some $Q_1>0$, then $\VMO_\mu^p(X)\subset\VMO^p(X)$. In particular, these two spaces coincide when $\mu$ is Ahlfors $Q$-regular, or when $\mu$ is a doubling measure on a bounded and connected metric space (this follows from taking $R=2\,\diam(X)$ in \cite[Lemma~3.3 and Corollary~3.8]{bb}).  In this section, which aims to show a relationship between vanishing mean oscillation and continuous functions, the choice of $\VMO(X)$ is more natural.
	
	
	
	
	
	\begin{proposition}\label{proposition:closed}
		For $1\le p<\infty$, $\VMO^p(X)$ is closed in $\BMO(X)$. 
	\end{proposition}
	
	\begin{proof}
		Take a sequence $\{f_n\}\subset\VMO^p(X)$ such that $f_n \rightarrow f$ in $\BMO^p(X)$. Fixing $n$, we have by the triangle inequality that for any ball $B\subset X$,
		$$
		\cO_p(f,B)\leq \cO_p(f_n,B)+\cO_p(f-f_n,B)\leq \cO_p(f_n,B)+\|f-f_n\|_{\BMO^p(X)}\,.
		$$
		Fixing $r>0$ and taking suprema over all balls satisfying $\rad(B)\leq r$, we have that 
		$$
		\omega_p(f,r) \leq \omega_p(f_n,r) + \|f-f_n\|_{\BMO^p(X)}\,.
		$$
		First, we send $r\rightarrow{0^+}$ and see that 
		$$
		\omega_p(f,0) \leq \omega_p(f_n,0) + \|f-f_n\|_{\BMO^p}=\|f-f_n\|_{\BMO^p(X)}
		$$
		since $f_n\in\VMO^p(X)$. Next, we send $n\rightarrow\infty$ and find that $\omega_p(f,0)=0$, giving that $f\in\VMO^p(X)$. 
		
		This demonstrates that $\VMO^p(X)$ is closed in $\BMO^p(X)$. Since $\BMO^p(X)\cong\BMO(X)$ for $1<p<\infty$, see Remark \ref{remark:john-nirenberg}, we also have that $\VMO^p(X)$ is closed in $\BMO(X)$.
	\end{proof}	
	
	In what follows, we denote by $C_u(X)$ the space of uniformly continuous functions on $X$. Any closures are meant to be taken with respect to the $\BMO(X)$-norm. 
	
	
	\begin{lemma}\label{prop:If}  
		For $1\le p<\infty,$ $\overline{ C_u(X)\cap\BMO(X)}\subset\VMO^p(X)$.
	\end{lemma}
	
	\begin{proof}
		Let $f\in C_u(X)\cap\BMO(X)$. From the uniform continuity of $f$, for any $\eps>0$ there exists $r>0$ such that $|f(x)-f(y)|<\eps$ whenever $d(x,y)<r$. For any ball $B$ of radius at most $r/2$, 
		$$
		\cO_p(f,B)\leq \left(\fint_{B}\fint_{B}\!|f(x)-f(y)|^p\,d\mu(x)\,d\mu(y)\right)^{1/p}<\eps,
		$$
		which implies, sending $\eps\rightarrow{0^+}$, that
		$$
		\omega_p(f,0)=\lim_{r\rightarrow{0^+}}\sup_{\rad(B)\leq r/2}\cO_p(f,B)=0.
		$$
		The result then follows from Proposition \ref{proposition:closed}.
	\end{proof}
	
	
	\begin{lemma}\label{lem:UC}
		For $f\in\BMO(X)$ and $\delta>0$, there exists a constant $C>0$, depending only on the doubling constant of $\mu$, such that the discrete convolution $f_\delta$ as defined in \eqref{discrete} is uniformly locally $\tfrac{C\,\|f\|_{\BMO(X)}}{\delta}$-Lipschitz in the sense that $d(x,y)<\delta$ implies that 
		\begin{equation}\label{Lip} 
			|f_\delta(x)-f_\delta(y)|\leq \frac{C}{\delta}\,\|f\|_{\BMO(X)}\,d(x,y).
		\end{equation}
		In particular, $f_\delta$ is uniformly continuous. 
	\end{lemma}
	
	\begin{proof}
		We may assume that $f$ is not constant almost everywhere; otherwise, $f_\delta$ is equal to that constant everywhere, and the conclusion follows.
		In particular, then we have $\|f\|_{\BMO}>0$.
		
		For $x,y\in X$ such that $d(x,y)<\delta$, set $I_{x,y}:=\{i\in\N:x\in 2B_i\;\text{or}\;y\in 2B_i$\}. By the bounded overlap property of the balls coming from Lemma \ref{lemma:cover} with $K=2$, we have that $I_{x,y}$ is a finite set with cardinality independent of $\delta$. By the construction of the partition of unity as in Lemma \ref{lemma:partition},  
		$$
		|f_\delta(x)-f_\delta(y)|=\left|\sum_{i\in I_{x,y}}(f_{B_i}-f_{B_{i_0}})(\varphi_i(x)-\varphi_i(y))\right|\leq \sum_{i\in I_{x,y}}|f_{B_i}-f_{B_{i_0}}||\varphi_i(x)-\varphi_i(y)|
		$$
		for any $i_0\in I_{x,y}$. 
		
		For $i,i_0\in I_{x,y}$, we can take a ball $\widetilde{B}$ of radius $4\delta$ that contains $B_i$ and $B_{i_0}$. From doubling, it follows that 
		$$
		|f_{B_i}-f_{\widetilde{B}}|\leq \fint_{B_i}\!|f-f_{\widetilde{B}}|\,d\mu \lesssim \fint_{\widetilde{B}}\!|f-f_{\widetilde{B}}|\,d\mu\leq \|f\|_{\BMO}
		$$
		and a similar inequality holding for $|f_{B_{i_0}}-f_{\widetilde{B}}|$. Thus, $|f_{B_i}-f_{B_{i_0}}|\lesssim \|f\|_{\BMO}$. Therefore, \eqref{Lip} follows from this, the Lipschitz property of $\varphi_i$, and the bounded overlap of the balls with index in $I_{x,y}$.
	\end{proof}
	
	\begin{lemma}\label{prop:OnlyIf}
		For $1\le p<\infty$, $\VMO^p(X)\subset\overline{C_u(X)\cap\BMO(X)}$.
	\end{lemma}
	
	
	\begin{proof}
		This lemma is proven by showing that there exists a constant $C\ge 1$, depending only on $C_\mu$, such that for all $f\in \BMO(X)$, $\dist(f,C_u(X)\cap\BMO(X))\le C\,\omega_p(f,0),$ where the distance is taken with respect to the $\BMO(X)$-norm.
		To this end, fix $f\in\BMO(X)$ and $\varepsilon>0.$  Choose $r>0$ such that 
		\begin{equation}\label{*}
			\sup_{\rad(B)\le 7r}\cO_p(f,B)<\omega_p(f,0)+\varepsilon.
		\end{equation}
		We then consider the discrete convolution of $f$ at the scale $\delta=r$, with associated cover $\{B_i\}$ and partition of unity $\{\varphi_i\}$, and aim to show that $\|f-f_\delta\|_{\BMO(X)}$ is controlled by $\omega_p(f,0)$.  
		
		Consider a ball $B\subset X$ with $\rad(B)\le \delta$. 
		For $x,y\in B,$ let $I_{x,y}:=\{i\in\N:x\in 2B_i\text{ or }y\in2B_i\}.$ For any $i_0\in I_{x,y},$ it follows from properties of the partition of unity, doubling, and H\"older's inequality that 
		\begin{align*}\label{eq:UniformContinuity}
			|f_\delta(x)-f_\delta(y)|&\le|f_\delta(x)-\sum_if_{B_{i_0}}\varphi_i(x)|+|\sum_if_{B_{i_0}}\varphi_i(y)-f_\delta(y)|\\
			&\le\sum_i\varphi_i(x)|f_{B_i}-f_{B_{i_0}}|+\sum_i\varphi_i(y)|f_{B_i}-f_{B_{i_0}}|\\
			&\lesssim \sum_{i\in I_{x,y}}|f_{B_i}-f_{B_{i_0}}|\\
			&\lesssim \sum_{i\in I_{x,y}}\fint_{7B_{i_0}}|f-f_{7B_{i_0}}|\,d\mu\\
			&\le \sum_{i\in I_{x,y}}\left(\fint_{7B_{i_0}}|f-f_{7B_{i_0}}|^p\,d\mu\right)^{1/p}\\
			&\lesssim \omega_p(f,0)+\varepsilon.
		\end{align*}
		The last inequality follows from the choice of $\delta$ satisfying \eqref{*} and bounded overlap. Indeed, the fact that $d(x,y)<2\delta$ implies that $I_{x,y}$ contains at most $C$ elements, for some constant $C\ge 1$ that depends on $C_\mu$ but is independent of $\delta$.  
		By H\"older's inequality and \eqref{*}, 
		\begin{equation}\label{***}
			\cO(f-f_\delta,B)<\omega_p(f,0)+\varepsilon+\fint_B\fint_B|f_\delta(x)-f_\delta(y)|\,d\mu(y)\,d\mu(x)\lesssim \omega_p(f,0)+\varepsilon
		\end{equation}
		for balls $B$ with $\rad(B)\le\delta.$
		
		We now consider the case where $B\subset X$ is a ball such that $\rad(B)>\delta.$  
		Let $I_B:=\{i\in\N:B\cap B_i\ne\varnothing\}$ and, for $i\in I_B,$ let $I_i:=\{j\in\N:2B_j\cap B_i\ne\varnothing\}.$ 
		From properties of the partition of unity, doubling, and H\"older's inequality, it follows that for $i\in I_B$,
		\begin{align*}
			\fint_{B_i}|f-f_\delta|\,d\mu=\fint_{B_i}|f-\sum_j f_{B_j}\varphi_j|\,d\mu&\le\sum_{j}\fint_{B_i}\varphi_j(x)|f(x)-f_{B_j}|\,d\mu(x)\\
			&\le\sum_{j\in I_i}\fint_{B_i}|f-f_{B_j}|\,d\mu\\
			&\lesssim \sum_{j\in I_i}\fint_{4B_i}|f-f_{4B_i}|\,d\mu\\
			&\lesssim \sum_{j\in I_i}\left(\fint_{4B_i}|f-f_{4B_i}|^p\,d\mu\right)^{1/p}\\
			&\lesssim \omega_p(f,0)+\varepsilon.
		\end{align*}
		Again, the last inequality follows from the choice of $\delta$ and the fact that there exists a constant $C\ge 1$, depending on $C_\mu$, such that $I_i$ contains at most $C$ elements by the bounded overlap property. 
		
		Note that $B\subset\bigcup_{i\in I_B}B_i\subset 2B,$ and so, by doubling and the fact that the collection $\{\frac{1}{5}B_i\}$ is disjoint, 
		\begin{align*}
			\fint_B|f-f_\delta|\,d\mu\le\sum_{i\in I_B}\frac{\mu(B_i)}{\mu(B)}\fint_{B_i}|f-f_\delta|\,d\mu&\lesssim\frac{\omega_p(f,0)+\varepsilon}{\mu(B)}\sum_{i\in I_B}\mu(B_i)\\&\lesssim \frac{\omega_p(f,0)+\varepsilon}{\mu(B)}\sum_{i\in I_B}\mu\left(\frac{1}{5}B_i\right)\\
			&\le\frac{\omega_p(f,0)+\varepsilon}{\mu(B)}\mu(2B)\\
			&\lesssim \omega_p(f,0)+\varepsilon.
		\end{align*}
		Hence, by \eqref{mean-constant},
		we have that 
		\begin{equation}\label{*****}
			\cO(f-f_\delta,B)\leq 2\fint_B|f-f_\delta|\,d\mu\lesssim \omega_p(f,0)+\varepsilon
		\end{equation}
		for balls $B$ such that $\rad(B)>\delta.$
		
		Combining \eqref{***} and \eqref{*****}, we have that
		\[
		\|f-f_\delta\|_{\BMO(X)}\lesssim  \omega_p(f,0)+\varepsilon.
		\]
		Since $f\in\BMO(X),$ the right-hand side is finite by Remark \ref{remark:john-nirenberg}, and so $f_\delta-f\in\BMO(X)$.  Hence $f_\delta=f+(f_\delta-f)\in\BMO(X).$
		Furthermore, since $f_\delta\in C_u(X)$ by Lemma~\ref{lem:UC}, it follows that 
		\[
		\dist(f, C_u(X)\cap\BMO(X))\lesssim \omega_p(f,0)+\varepsilon,
		\]
		and so taking $\varepsilon\to 0$ yields the desired result.
	\end{proof}
	
	Combining Lemmas~\ref{prop:If} and~\ref{prop:OnlyIf} yields the following characterization of $\VMO^p(X)$.
	
	\begin{theorem}\label{thm:VMOChar}
		For $1\le p<\infty$, $\VMO^p(X)=\overline{C_u(X)\cap\BMO(X)}$.
	\end{theorem}
	
	An immediate consequence of this result is the following.
	
	\begin{corollary}\label{cor:VMOp-VMO}
		For $1\le p<\infty$, $\VMO^p(X)=\VMO(X)$. 
	\end{corollary}
	

	
	\section{The Fractional Maximal Function Acting on Function Spaces}\label{sec:fs}
	
	As in the previous section, $(X,d,\mu)$ is a doubling metric measure space with doubling constant $C_\mu$. 
	For a function $f\in\Loneloc(X)$ and $\alpha\geq 0$, define the following maximal function:
	$$
	M^\alpha f(x)=\sup_{B\ni x}\rad(B)^\alpha\,\fint_{B}\!|f|\,d\mu\,,
	$$
	where the supremum is taken over all balls $B\subset X$ containing the point $x\in X$ satisfying $0<\rad(B)<2\,\diam(X)$. When $\alpha=0$, $Mf:=M^0f$ is the standard (uncentered) Hardy-Littlewood maximal function. When $\alpha>0$, then $M^\alpha f$ is known as the fractional maximal function. 
	It is easy to see that $M^\alpha f$ is lower semicontinuous and therefore measurable. For a general locally integrable function, though, nothing precludes $M^\alpha f$ from being identically infinite, even when $\alpha=0$. 
	
	For $\alpha>0$, we will require the additional assumption that $(X,d,\mu)$ satisfies a lower mass bound with exponent $Q$, see \eqref{LMB}, and will restrict the definition of $M^\alpha$ to $0<\alpha<Q$. For $\alpha=0$, the measure being doubling suffices. 
	
	An important feature of $M^\alpha$ is that it is pointwise sublinear. That is, for $f,g\in\Loneloc(X)$ and $x\in X$, $M^\alpha(f+g)(x)\leq M^\alpha f(x)+M^\alpha g(x)$. It follows from this that
	$$
	M^\alpha f(x)\leq M^\alpha(f-g)(x)+M^\alpha g(x)
	$$
	and 
	$$
	M^\alpha g(x)\leq M^\alpha(g-f)(x)+M^\alpha f(x)=M^\alpha(f-g)(x)+M^\alpha f(x),
	$$ and so, if $M^\alpha f(x)<\infty$ and $M^\alpha g(x)<\infty$, then
	\begin{equation}\label{contract}
		|M^\alpha f(x)-M^\alpha g(x)|\leq M^\alpha (f-g)(x). 
	\end{equation}
	It follows from Proposition \ref{prop:LocBound} that this inequality holds almost everywhere for $f,g\in L^p(X)$ for some $1\leq p<\infty$.
	
	\subsection{Lebesgue spaces}
	
	
	In order to study the behaviour of $M^\alpha$ on $\BMO(X)$ and $\VMO(X)$, we must first know its behaviour on $L^p(X)$ spaces. This is the content of the following theorem,
	which will be required in a subsequent section, the proof of which can be found in \cite[Section 3]{HKNT}. Their proof is for a slightly different fractional maximal function that is equivalent to the one discussed here in a doubling metric measure space; we reproduce the argument here in our precise setting for the benefit of the reader. 
	
	
	
	
	
	For simplicity of notation, we define a parameter $p^*$ that depends on $p,\alpha$, and $Q$. For $0\leq \alpha\leq Q/p$ and $1< p< \infty$, define $p\leq p^*\leq\infty$ by 
	\begin{equation}\label{eq:p-q}
		\frac{1}{p}-\frac{1}{p^*}=\frac{\alpha}{Q}\,.
	\end{equation}
	
	\begin{proposition}\label{prop:LocBound}
		Let $(X,d,\mu)$ be a doubling metric measure space, $\alpha\geq 0$, and $1\leq p<\infty$. When $\alpha>0$, we assume that $(X,d,\mu)$ satisfies a lower mass bound with exponent $Q$. 
		
		If $p=1$ and $0\leq \alpha<Q$, then
		$$
		\mu(\{x\in X: M^\alpha f(x)>t\})^{\frac{Q-\alpha}{Q}} \leq  \frac{C\,\|f\|_{L^1(X)}}{t}
		$$
		for all $f\in L^1(X)$, where $C\ge 1$ depends on $C_\mu$ and, if $0<\alpha<Q$, on $c_L$ and $\alpha$. That is, $M^\alpha$ is weak-type $(1,\frac{Q}{Q-\alpha})$. 
		
		If $1<p<\infty$ and $0\leq  \alpha\leq  Q/p$, then 
		$$
		\|M^\alpha f\|_{L^{p^*}(X)}\leq C\, \|f\|_{L^p(X)}\,
		$$
		for all $f\in L^p(X)$, where $C\ge 1$ depends on $C_\mu$ and $p$, and, if $0<\alpha<Q$, on $c_L$ and $\alpha$. That is, $M^\alpha$ is strong-type $(p,p^*)$. 
	\end{proposition}
	
	\begin{proof}
		The case $\alpha=0$ is the Hardy-Littlewood-Weiner theorem. 
		
		For the case $\alpha>0$, take $f\in L^p(X)$ with $\|f\|_{L^p(X)}>0$, where for now $p$ may be anything in the interval $[1,\infty)$. We begin by deriving a pointwise inequality relating $M^\alpha f$ and $Mf$, the standard Hardy-Littlewood maximal function on $X$. To this end, fix $x\in X$ such that $Mf(x)<\infty$, and select a ball $B\ni x$. If $r$ is the radius of $B$, then an application of H\"older's inequality and the lower mass bound on $\mu$, see \eqref{LMB}, yields
		\begin{align*}
			r^\alpha \fint_{B}\!|f|\,d\mu&=r^\alpha\left(\fint_{B}\!|f|\,d\mu\right)^{\alpha p/Q}\left(\fint_{B}\!|f|\,d\mu\right)^{1-\alpha p/Q}
			\\&\leq \frac{r^\alpha}{\mu(B)^{\alpha p/Q}}\left(\int_{B}\!|f|\,d\mu\right)^{\alpha p / Q}\!Mf(x)^{1-\alpha p/Q}\\&\leq \frac{r^\alpha}{\mu(B)^{\alpha p/Q}}\left(\int_{X}\!|f|^p\,d\mu\right)^{\alpha / Q}\!\mu(B)^{(1-1/p)\alpha p/Q}\!Mf(x)^{1-\alpha p/Q}\\&= \left(\frac{r}{\mu(B)^{1/Q}}\right)^\alpha\|f\|_{L^p(X)}^{\alpha p /Q}Mf(x)^{1-\alpha p/Q}\\&\lesssim \|f\|_{L^p(X)}^{\alpha p /Q}Mf(x)^{1-\alpha p/Q}.
		\end{align*}
		The very first inequality is where we use that $0<\alpha\leq Q/p$ if $p\in(1,\infty)$ and $0<\alpha<Q$ if $p=1$. Taking a supremum over all $B\ni x$, we obtain the estimate 
		\begin{equation}\label{eq:pointwisecomparison}
			M^\alpha f(x)\lesssim \|f\|_{L^p(X)}^{\alpha p /Q}Mf(x)^{1-\alpha p/Q}.
		\end{equation}
		
		When $p=1$, this implies that
		$$
		\mu(\{M^\alpha f(x)>t\})\leq \mu\left(\left\{M f(x)> \frac{t^{\frac{Q}{Q-\alpha}}}{ C\,\|f\|_{L^1(X)}^{\frac{\alpha}{Q-\alpha}}}\right\}\right)\lesssim \frac{C\|f\|_{L^1(X)}^{\frac{\alpha}{Q-\alpha}}\|f\|_{L^1(X)}}{t^{\frac{Q}{Q-\alpha}}}=\left(\frac{C\,\|f\|_{L^1(X)}}{t}\right)^{\frac{Q}{Q-\alpha}}
		$$
		for any $0<\alpha<Q$. The boundedness of $M$ from $L^1(X)$ to weak-$L^1(X)$ was used in the second step. Raising both sides to the power of $\frac{Q-\alpha}{Q}$ gives the first result.
		
		For $p>1$, \eqref{eq:pointwisecomparison} implies that 
		$$
		\|M^\alpha f\|_{L^{p^*}(X)}^{p^*}\lesssim \|f||_{L^p(X)}^{\alpha p{p^*}/Q}\int_{X}\!(Mf)^{{p^*}-\alpha p{p^*}/Q}\,d\mu = \|f\|_{L^p(X)}^{{p^*}-p}\|Mf\|_{L^p(X)}^p\lesssim \|f\|_{L^p(X)}^{{p^*}},
		$$
		for any $0<\alpha\leq Q/p$, where ${p^*}$ is as in \eqref{eq:p-q}. The last step followed from the boundedness of $M$ on $L^p(X)$. The second result of the proposition follows from raising both sides to the power of $1/p^*$.
	\end{proof}
	
	\subsection{Bounded mean oscillation}
	
	In the Euclidean setting, it was first proven that $M$, taken with respect to cubes, is bounded from $\BMO$ to $\BLO$ \cite{ben} after the weaker statement that it is bounded from $\BMO$ to itself was shown in \cite{bdvs}. For alternative proofs and results for related maximal functions, see \cite{ak, cun,cf,ler,ou,saari}. The proof of the following theorem follows from an adaptation of the proof in the Euclidean case in \cite{ben}, as will be demonstrated below.

	\begin{theorem}\label{theorem:BMO}
		Let $(X,d,\mu)$ be a doubling metric measure space and $\alpha\geq 0$. When $\alpha>0$, we assume that $(X,d,\mu)$ is bounded with lower mass bound exponent $Q$. If $f\in\BMO(X)$ and $0\leq\alpha<Q$, then
		\begin{equation}\label{BLO}
			\fint_{B}\!M^\alpha f\,d\mu \leq C\left(\mu(X)^{\alpha/Q}+\diam(X)^\alpha\right)\,\|f\|_{\BMO(X)}+\essinf_{B}M^\alpha f,
		\end{equation}
		for all balls $B\subset X$, where $C>0$ depends on $C_\mu$ and $\alpha$. 
		
		In particular, if $\alpha=0$ and $M^\alpha f$ is not identically zero, or if $0<\alpha<Q$, it follows that
		$M^\alpha f \in\BLO(X)$ with
		$$
		\|M^\alpha f\|_{\BLO(X)}\leq C\left( \mu(X)^{\alpha/Q}+\diam(X)^\alpha \right)\,\|f\|_{\BMO(X)}. 
		$$
	\end{theorem}
	
	
	\begin{remark}
		When $\mu$ is doubling and satisfies an annular decay property, a proof of the $\alpha=0$ case can be found in \cite{kurk}, where the author uses arguments related to Muckenhoupt weights. The proof in the present paper improves upon this result by showing that the annular decay property is not necessary. Note that annular decay {\emph{is}} necessary to show that $M^\alpha$ maps $\VMO(X)$ to itself, see the next subsection.
	\end{remark}
	
	
	\begin{proof}[Proof of Theorem \ref{theorem:BMO}]
		Fix $0\leq \alpha<Q$, $f\in\BMO(X)$, and a ball $B\subset X$. We decompose $f=g+h$, where $g=(f-f_{2B})\chi_{2B}$ and $h=f_{2B}\chi_{2B}+f\chi_{(2B)^c}$. 
		
		We begin by estimating the mean of $M^\alpha g$ on $B$. Set $p=2Q/(Q+\alpha)\in (1,2]$ so that $p^*=2Q/(Q-\alpha)\in [2,\infty)$. Then $0\leq \alpha\leq Q/p$, and so we may apply Proposition \ref{prop:LocBound} along with H\"older's inequality to see that
		\begin{equation}\label{local}
			\begin{aligned}
				\fint_{B}\!M^\alpha g\,d\mu \leq \frac{1}{\mu(B)^{1/{p^*}}}\|M^\alpha g\|_{L^{p^*}(X)}
				\lesssim \frac{1}{\mu(B)^{1/{p^*}}}\|g\|_{L^p(X)}
				& =\mu(B)^{\alpha/Q}\left(\fint_{2B}\!|f-f_{2B}|^p\,d\mu \right)^{1/p}\\&\lesssim \mu(B)^{\alpha/Q}\|f\|_{\BMO(X)}\,,
			\end{aligned}
		\end{equation}
		where we used Remark \ref{remark:john-nirenberg} in the last step.
		
		Next we obtain a pointwise bound on $M^\alpha h$ on $B$. To this end, fix $x_0\in B$ and another ball $\widetilde{B}$ such that $\widetilde{B}\ni x_0$. Denote by $r$ the radius of $B$ and by $\tilde{r}$ the radius of $\widetilde{B}$. If $\widetilde{B}\subset 2B$, then $\tilde{r}\leq 2r$ and
		\begin{equation}\label{local-global}
			\tilde{r}^\alpha\fint_{\widetilde{B}}\!|h|\,d\mu = \tilde{r}^\alpha|f_{2B}|\leq (2r)^\alpha\fint_{2B}|f|\,d\mu\leq M^\alpha f(x)
		\end{equation}
		for any $x\in B$. If $\widetilde{B}\cap (2B)^c\neq\emptyset$, we consider $8\widetilde{B}$, which contains both $\widetilde{B}$ and $2B$ and, by the doubling property, satisfies $\mu(8\widetilde{B})\leq C_\mu^3 \mu(\widetilde{B})$. Then, by the triangle inequality,
		$$
		\tilde{r}^\alpha\fint_{\widetilde{B}}\!|h|\,d\mu\leq \tilde{r}^\alpha\fint_{\widetilde{B}}\!|h-f_{8\widetilde{B}}|\,d\mu+ \tilde{r}^\alpha\fint_{8\widetilde{B}}\!|f|\,d\mu\leq C_\mu^3\,\tilde{r}^\alpha\fint_{8\widetilde{B}}\!|h-f_{8\widetilde{B}}|\,d\mu + M^\alpha f(x)\,,
		$$
		again, where $x$ is any point in $B$. Writing $8\widetilde{B}$ as the union of $2B$ and $8\widetilde{B}\cap(2B)^c$, we have that
		\begin{align*}
			\int_{8\widetilde{B}}\!|h-f_{8\widetilde{B}}|\,d\mu &= \int_{2B}\!|h-f_{8\widetilde{B}}|\,d\mu + \int_{8\widetilde{B}\cap(2B)^c}\!|h-f_{8\widetilde{B}}|\,d\mu\\&= \int_{2B}\!|f_{2B}-f_{8\widetilde{B}}|\,d\mu + \int_{8\widetilde{B}\cap(2B)^c}\!|f-f_{8\widetilde{B}}|\,d\mu\\&\leq \int_{2B}\!|f-f_{8\widetilde{B}}|\,d\mu+\int_{8\widetilde{B}\cap(2B)^c}\!|f-f_{8\widetilde{B}}|\,d\mu=\int_{8\widetilde{B}}\!|f-f_{8\widetilde{B}}|\,d\mu\leq\mu(8\widetilde{B})\|f\|_{\BMO(X)}\,,
		\end{align*}
		and so it follows that
		\begin{equation}\label{global-global}
			\tilde{r}^\alpha\fint_{\widetilde{B}}\!|h|\,d\mu\leq C_\mu^3\diam(X)^\alpha \|f\|_{\BMO(X)}+M^\alpha f(x)
		\end{equation}
		Combining equations~\eqref{local-global} and~\eqref{global-global} yield a pointwise estimate for $M^\alpha h(x_0)$
		$$
		M^\alpha h(x_0)\leq C_\mu^3\diam(X)^\alpha \|f\|_{\BMO(X)}+M^\alpha f(x)\,.
		$$
		Taking the essential infimum over all $x\in B$ and then taking an average over all $x_0\in B$ yields 
		\begin{equation}
			\fint_{B}\!M^\alpha h\,d\mu \leq C_\mu^3\diam(X)^\alpha \|f\|_{\BMO(X)}+\essinf_{B}M^\alpha f\,. 
		\end{equation}
		Combining this with~\eqref{local}, the pointwise sublinearity of the fractional maximal function implies the desired inequality~\eqref{BLO}. 
	\end{proof}
	
	\subsection{Vanishing mean oscillation}
	
	In this subsection, we consider the action of $M^\alpha$ on the subspace $\VMO(X)$. As the norm on $\VMO(X)$ is the same as that of $\BMO(X)$, we already know that the $\BMO$-norm of the image of a $\VMO(X)$ function under $M^\alpha$ is controlled. The goal here is to show that $M^\alpha$ preserves the vanishing mean oscillation condition itself.
	
	The first result in this direction can be found in \cite{shahab}, where the author considers $\alpha=0$ on $X=\R^n$ with the Euclidean metric and Lebesgue measure. We generalize this to other values of $\alpha$ and to a class of doubling metric measure spaces. 
	
	To accomplish this in the following theorem, we impose an additional hypothesis compared to Theorem \ref{theorem:BMO}: an annular decay condition, see Definition \ref{def:ann-decay}. This is a necessary assumption for the theorem to hold, as follows from \cite[Example 5.2]{hlnt}. Modifying this example by restricting $X$ to be bounded yields a bounded doubling metric measure space with lower mass bound exponent 1 that does not satisfy an annular decay property. On this space, the authors consider a bounded Lipschitz (hence $\VMO$) function $f$ such that $M^\alpha f$ has a jump discontinuity and so cannot be in $\VMO$.
	
	\begin{theorem}\label{theorem:VMO}
		Let $(X,d,\mu)$ be a doubling metric measure space satisfying a $\beta$-annular decay property for $0<\beta\leq 1$, and $\alpha\geq 0$. When $\alpha>0$, we assume that $(X,d,\mu)$ is bounded with lower mass bound exponent $Q$. Let $f\in\VMO(X)$. If $\alpha=0$ and $M^\alpha f$ is not identically infinite, or if $0<\alpha<Q$, then $M^\alpha f\in\VMO(X)$. 
	\end{theorem}

	
	We note that in the case of small positive $\alpha$ (in particular $0<\alpha<\beta$) and $X$ compact, this result can be obtained as a corollary of \cite[Theorem 3.1]{hlnt}. Indeed, it is shown that for any $f\in\BMO(X)$, $M^\alpha f$ is continuous. In the compact setting, continuous functions are automatically uniformly continuous and bounded, hence in BMO, and so $M^\alpha f$ is in $\VMO(X)$ by Lemma \ref{prop:If} and Corollary \ref{cor:VMOp-VMO}. This is stronger than Theorem \ref{theorem:VMO} as it shows that all $\BMO(X)$ functions map to $\VMO(X)$, but Theorem \ref{theorem:VMO} has the benefit of not requiring compactness of $X$ and holding for a larger interval of $\alpha$. In particular, one could apply it to an open bounded domain in $\R^n$. The proof of Theorem \ref{theorem:VMO} can be found at the end of the present subsection.
	
	
	Fix a ball $B\subset X$ with $\rad(B)=r$ and a scale $\lambda\geq 1$. We write 
	$$
	M^\alpha f=\max\{ M^\alpha_{\loc,\lambda\,r}f,M^\alpha_{\glob,\lambda\,r}f \},
	$$
	where $M^\alpha_{\loc,\lambda\,r}f$ is the ``local" part of $M^\alpha f$ given by 
	$$
	M^\alpha_{\loc,\lambda\,r}f(x)=\sup_{\widetilde{r}< \lambda\,r} \widetilde{r}^\alpha\fint_{B(\widetilde{r})}|f|\,d\mu,
	$$
	where the supremum is taken over all balls $B(\widetilde{r})$ with radius $\widetilde{r}<\lambda\,r$, and $M^\alpha_{\glob,\lambda\,r}f$ is the ``global" counterpart given by 
	$$
	M^\alpha_{\glob,\lambda\,r}f(x)=\sup_{\widetilde{r}\geq \lambda\,r} \widetilde{r}^\alpha\fint_{B(\widetilde{r})}|f|\,d\mu,
	$$
	where the supremum is taken over all balls $B(\widetilde{r})$ with radius $\widetilde{r}\geq \lambda\,r$. The goal is to estimate the dependence on $\lambda$ of the mean oscillation of the local and global parts. We begin with the local part. 
	
	\begin{lemma}\label{lem:LocOsc} 
		Let $(X,d,\mu)$ be a doubling metric measure space and $\alpha\geq 0$. If $\alpha>0$, we assume additionally that $(X,d,\mu)$ is bounded with lower mass bound exponent $Q$. Given a ball $B\subset X$ with $\rad(B)=r$, $\lambda\ge 1$, $f\in L^1_{\loc}(X)$, and $0\leq\alpha<Q$, there exists a constant $C$, depending on $C_\mu$ and $\alpha$, such that
		\[
		\cO(M^{\alpha}_{\loc,\lambda\,r}f,B)\le C\lambda^{Q/p}\mu(B)^{\alpha/Q}\cO_p(f,3\lambda B),
		\]
		where $p=2Q/(Q+\alpha)$. 
	\end{lemma}
	
	\begin{proof}
		For simplicity of notation, we suppress the dependence of $M^{\alpha}_{\loc,\lambda\,r}f$ on $\lambda$ and $r$. By \eqref{mean-constant}
		and the definition of $M^\alpha_{\loc}$, we have that 
		$$
		\cO(M^\alpha_{\loc}f,B)\le 2\fint_{B}\!|M^\alpha_{\loc} f-(\lambda\,r)^\alpha |f_{\lambda B}||\,d\mu= 2\fint_{B}\!|M^\alpha_{\loc} f-M_{\loc}^\alpha(f_{\lambda B})|\,d\mu.
		$$
		Notice that $M^\alpha_{\loc} f=M^\alpha_{\loc} (f\chi_{3\lambda B})$ and $M^\alpha_{\loc} f_{\lambda B}=M^\alpha_{\loc} (f_{\lambda B}\chi_{3\lambda B})$ on $B$ as points outside of $3\lambda B$ cannot be reached by balls of radius at most $\lambda r$. Therefore, by \eqref{contract}, we have that
		$$
		\cO(M^\alpha_{\loc}f,B)\leq 2\fint_{B}\!|M^\alpha_{\loc} (f\chi_{3\lambda B})-M_{\loc}^\alpha(f_{\lambda B}\chi_{3\lambda B})|\,d\mu\leq 2\fint_{B}\!M^\alpha_{\loc} [(f-f_{\lambda B})\chi_{3\lambda B}]\,d\mu.
		$$
		
		Write $g=(f-f_{\lambda B})\chi_{3\lambda B}$. Now, for $p$ as in the statement of the lemma, we have that $0\leq\alpha\leq Q/p$, and so we may apply Proposition~\ref{prop:LocBound} and H\"{o}lder's inequality to see that
		$$
		\cO(M^\alpha_{\loc}f,B)\leq \frac{1}{\mu(B)^{1/p^*}}\|M^\alpha g\|_{L^{p^*}(X)}\lesssim \frac{1}{\mu(B)^{1/p^*}}\|g\|_{L^{p}(X)}= \frac{\mu(3\lambda B)^{1/p}}{\mu(B)^{1/p^*}}\left(\fint_{3\lambda B}\!|f-f_{\lambda B}|^p\,d\mu\right)^{1/p}\!.
		$$
		Recall that $p^*=\frac{pQ}{Q-\alpha p}$. From the doubling of $\mu$, it follows that $\frac{\mu(3\lambda B)^{1/p}}{\mu(B)^{1/p^*}}\lesssim \lambda^{Q/p}\mu(B)^{1/p-1/p^*}=\lambda^{Q/p}\mu(B)^{\alpha/Q}$. Also from doubling it follows that
		\begin{align*}
			\fint_{3\lambda B}\!|f-f_{\lambda B}|^p\,d\mu&\lesssim \fint_{3\lambda B}\fint_{\lambda B}\!|f(x)-f(y)|^p\,d\mu(y)d\mu(x)\\ &\lesssim \fint_{3\lambda B}\fint_{3\lambda B}\!|f(x)-f(y)|^p\,d\mu(y)d\mu(x)\lesssim \fint_{3\lambda B}\!|f-f_{3\lambda B}|^p\,d\mu.
		\end{align*}
		Therefore, as desired, 
		$$
		\cO(M^\alpha_{\loc}f,B)\lesssim \lambda^{Q/p}\mu(B)^{\alpha/Q}\, \cO_p(f,3\lambda B)\,.
		$$
		
	\end{proof}
	
	
	%
	
	Now we continue to estimating the mean oscillation of the global part of the fractional maximal function. The following proof is inspired by \cite[Theorem 3.1]{hlnt} and the techniques go back as far as \cite[Section 3]{McMan}.
	
	\begin{lemma}\label{lem:NonLocOsc}
		
		Let $(X,d,\mu)$ be a doubling metric measure space and $\alpha\geq 0$. If $\alpha>0$, we assume additionally that $(X,d,\mu)$ is bounded with lower mass bound exponent $Q$. Given a ball $B\subset X$ with $\rad(B)=r$, $\lambda\ge 4$, $f\in \BMO(X)$, and $0\leq\alpha<Q$, there exists a constant $C$, depending on $C_\mu$, $C_\beta$, $\alpha$, and $\beta$ such that
		\[
		\cO(M^{\alpha}_{\glob,\lambda\,r}f,B)\le C\diam(X)^\alpha(\lambda^{-\beta}\log(C\lambda)+\lambda^{-\beta})\|f\|_{\BMO(X)}.
		\]
	\end{lemma}
	
	\begin{proof}
		As before, we suppress the dependence of the global fractional maximal function on the scale $\lambda$ and $r$. 
		
		Fix $\eps>0$ and consider $x,y\in B$. Without loss of generality, assume that $M_{\glob}^\alpha f(x)\ge M_{\glob}^\alpha f(y).$  Let $B_x=B(\xtilde,\rtilde)\subset X$ be such that $x\in B_x$, $\rtilde\geq \lambda r,$ and 
		\[
		M^\alpha_{\glob}f(x)\le \rtilde^\alpha\fint_{B_x}|f|\,d\mu+\varepsilon.
		\]
		Now let $B_y=B(\xtilde,\rtilde+2r).$  Then $y\in B_y$, and so we have that 
		$$
		M^\alpha_{\glob}f(x)-M^\alpha_{\glob}f(y)\le \rtilde^\alpha\fint_{B_x}\!|f|\,d\mu+\eps-(\rtilde+2r)^\alpha\fint_{B_y}\!|f|\,d\mu\le \rtilde^\alpha(|f|_{B_x}-|f|_{B_y})+\varepsilon.
		$$

		By doubling and the $\beta$-annular decay property,
		$$
		\frac{\mu(B_y\setminus B_x)}{\mu(B_x)}\lesssim \Big(\frac{2r}{\widetilde r}\Big)^\beta\frac{\mu(B_y)}{\mu(B_x)}\lesssim \Big(\frac{2r}{\widetilde r}\Big)^\beta
		$$
		and so 
		\begin{align*}
			|f|_{B_x}-|f|_{B_y}&=\frac{1}{\mu(B_x)}\left(\int_{B_x}|f|\,d\mu-\int_{B_y}|f|\,d\mu+(\mu(B_y)-\mu(B_x))|f|_{B_y}\right)\\
			&=\frac{1}{\mu(B_x)}\int_{B_y\setminus B_x}(|f|_{B_y}-|f|)\,d\mu\\
			&\le\frac{1}{\mu(B_x)}\int_{B_y\setminus B_x}\fint_{B_y}|f(z)-f_{B_y}|+|f(w)-f_{B_y}|d\mu(z)d\mu(w)\\
			&\leq \frac{1}{\mu(B_x)}\Bigg(\int_{B_y\setminus B_x}|f-f_{B_y}|\,d\mu+\mu(B_y\setminus B_x)\fint_{B_y}|f-f_{B_y}|d\mu\Bigg)\\
			&\le\frac{1}{\mu(B_x)}\int_{B_y\setminus B_x}|f-f_{B_y}|\,d\mu+C\Big(\frac{2r}{\widetilde r}\Big)^\beta\|f\|_{\BMO(X)}
		\end{align*}
		for some constant $C$.
		Therefore, it follows that that 
		\begin{equation}\label{eq:MacManusTrick}
			M^\alpha_{\glob}f(x)-M^\alpha_{\glob}f(y)\le\rtilde^\alpha\Bigg(\frac{1}{\mu(B_x)}\int_{B_y\setminus B_x}|f-f_{B_y}|\,d\mu+C\Big(\frac{r}{\widetilde r}\Big)^\beta\|f\|_{\BMO(X)}\Bigg)+\varepsilon.
		\end{equation}
		
		Now let $B_1,\dots,B_k$ be a maximal collection of disjoint balls of radius $r$ with centers in $B_y\setminus B_x.$  Then, 
		\[
		B_y\setminus B_x\subset\bigcup_i2B_i\subset\Delta:=B(\xtilde,\rtilde+4r)\setminus B(\xtilde,\rtilde-2r)\,.
		\]
		Note that we have insisted that $\lambda\ge4$, which ensures that $\widetilde r>4r.$  This in turn ensures that the annulus $\Delta$ above is sufficiently thin.
		
		For $i\in\{1,\dots,k\},$ we have from Lemma \ref{lem:jones} that 
		\begin{align}\label{eq:LogTerm}
			\fint_{2B_i}|f-f_{B_y}|\,d\mu&\le\fint_{2B_i}|f-f_{2B_i}|\,d\mu+|f_{2B_i}-f_{2B_y}|+|f_{2B_y}-f_{B_y}|\nonumber\\
			&\le\|f\|_{\BMO(X)}+C\log\left(\frac{C\rtilde}{r}\right)\|f\|_{\BMO(X)}\le C\log\left(\frac{C\rtilde}{r}\right)\|f\|_{\BMO(X)}.
		\end{align}
		
		By the $\beta$-annular decay property and doubling, we also have that 
		\begin{equation}\label{eq:AnnularDecayTerm}
			\mu(\Delta)\le C\mu(B(\xtilde,\rtilde+4r))\left(\frac{6r}{\rtilde+4r}\right)^\beta\lesssim \mu(B_x)\left(\frac{r}{\rtilde}\right)^\beta.
		\end{equation}
		
		Using \eqref{eq:LogTerm}, doubling, disjointness of $\{B_i\},$ and \eqref{eq:AnnularDecayTerm}, we come to
		\begin{align*}
			\frac{1}{\mu(B_x)}\int_{B_y\setminus B_x}|f-f_{B_y}|d\mu&\le\frac{1}{\mu(B_x)}\sum_{i}\mu(2B_i)\fint_{2B_i}|f-f_{B_y}|\,d\mu\\
			&\le \frac{C}{\mu(B_x)}\log\left(\frac{C\rtilde}{r}\right)\|f\|_{\BMO(X)}\sum_{i}\mu(2B_i)\\
			&\lesssim \frac{C}{\mu(B_x)}\log\left(\frac{C\rtilde}{r}\right)\|f\|_{\BMO(X)}\sum_{i}\mu(B_i)\\
			&\le \frac{C}{\mu(B_x)}\log\left(\frac{C\rtilde}{r}\right)\|f\|_{\BMO(X)}\,\mu(\Delta)\\
			&\lesssim C\left(\frac{r}{\rtilde}\right)^\beta\log\left(\frac{C\rtilde}{r}\right)\|f\|_{\BMO(X)}.
		\end{align*}
		Note that the function $t\mapsto t^{-\beta}\log(Ct)$ is decreasing in $t$, and so $\rtilde>\lambda r$ implies that
		$$
		\left(\frac{r}{\rtilde}\right)^\beta\log\left(\frac{C\rtilde}{r}\right)\leq \lambda^{-\beta}\log(C\lambda).
		$$ 
		Likewise, $(r/\rtilde)^\beta<\lambda^{-\beta}.$
		Therefore, combining the above estimate with \eqref{eq:MacManusTrick}, it follows that 
		\[
		|M^\alpha_{\glob}f(x)-M^\alpha_{\glob}f(y)|\le C\rtilde^\alpha\|f\|_{\BMO(X)}(\lambda^{-\beta}\log(C\lambda)+\lambda^{-\beta})+\eps.
		\] 
		Recall that without loss of generality, we have assumed $M^\alpha_{\glob}f(x)\ge M^\alpha_{\glob}f(y).$  Hence, taking $\eps\to 0,$ we obtain
		\begin{align*}
			\cO(M_{\glob}^{\alpha}f,B)&\le\fint_{B}\fint_{B}\!|M_{\glob}^{\alpha}f(x)-M_{\glob}^{\alpha}f(y)|\,d\mu(y)d\mu(x)\\
			&\leq C\diam(X)^\alpha(\lambda^{-\beta}\log(C\lambda)+\lambda^{-\beta})\|f\|_{\BMO(X)}.\qedhere
		\end{align*}
	\end{proof}
	
	With the separate estimates for the local and global parts of the fractional maximal function in hand, we are now able to provide a quantitative estimate of the mean oscillation of $M^\alpha f$ in terms of a scale $\lambda$. This is the essence of the following proof.
	
	\begin{proof}[Proof of Theorem \ref{theorem:VMO}]
		Let $f\in\VMO(X)$ and fix a ball $B\subset X$ with $\rad(B)=r$. If $\alpha=0$, we assume that $M^\alpha f$ is not identically infinite. For any $\lambda\geq 1$, we have that 
		\[M^\alpha f=\max\left\{M^{\alpha}_{\loc,\lambda r}f,M^{\alpha}_{\glob,\lambda r}f\right\},\] and so it follows that 
		\begin{equation}\label{max-osc}
			\cO(M^\alpha f,B)\leq \cO(M^\alpha_{\loc,\lambda r} f,B)+\cO(M^\alpha_{\glob,\lambda r} f,B),
		\end{equation}
		see, for instance, \cite[Proposition 6.2]{dg}. 
		
		Let $\eps>0$. We begin with the global part. Using Lemma \ref{lem:NonLocOsc} and the fact that $t^{-\beta}\log(Ct)\to0$ and $t^{-\beta}\to 0$ as $t\to\infty$, we can choose $\lambda\geq 4$ sufficiently large such that
		\[
		\cO(M^\alpha_{\glob,\lambda r} f,B)\leq C\diam(X)^\alpha(\lambda^{-\beta}\log(C\lambda)+\lambda^{-\beta})\|f\|_{\BMO(X)}<\frac{\varepsilon}{2}.
		\]
		Note that this choice of $\lambda\ge 4$ is independent of $B$.
		
		For the local part, we use the fact that $f\in\VMO(X)$, and hence in $\VMO^p(X)$ by Corollary \ref{cor:VMOp-VMO} with $p=2Q/(Q+\alpha)\in (1,2]$, to claim that there exists a radius $\rho>0$ such that 
		$$
		\cO_p(f,\widetilde{B})<\frac{\varepsilon}{2C\lambda^{Q/p}\mu(X)^{\alpha/Q}}
		$$
		for any ball $\widetilde{B}$ with radius $\rad(\widetilde{B})<\rho$. Here $C$ is the constant from Lemma \ref{lem:LocOsc}. Thus, by the very same lemma, if $\rad(B)<\rho/(3\lambda)$, then
		$$
		\cO(M^\alpha_{\loc,\lambda r} f,B)\leq C\lambda^{Q/p}\mu(X)^{\alpha/Q}\,\cO_p(f,3\lambda B)<\frac{\eps}{2}\,.
		$$
		
		
		Using \eqref{max-osc}, now, we have that for all $\eps>0$ there exists a radius $\widetilde{r}=\rho/(3\lambda)>0$ so that for any ball $B$ with $\rad(B)<\widetilde{r}$, we have that $\mathcal{O}(M^\alpha f,B)<\eps$. This is to say that $M^\alpha f\in\VMO(X)$.

	\end{proof}
	
	\subsection{Euclidean space}
	
	In this subsection, we consider the case when $X$ is a a finite cube $Q_0\subset\R^n$ with sides parallel to the axes. We consider $\R^n$ with the Euclidean metric and Lebesgue measure. 
	
	The space $\BMO(Q_0)$ has two possible definitions. The definition as in this paper is given by a bounded mean oscillation condition over all sets of the form $B(x,r)\cap Q_0$, where $x\in Q_0$ and $B(x,r)$ is the Euclidean ball centered at $x$ of radius $r$. This space can easily be seen to be isomorphic, up to a constant depending on $n$, to the one given by a bounded mean oscillation condition over all sets of the form $Q(x,\ell)\cap Q_0$, where $x\in Q_0$ and $Q(x,\ell)$ is the cube centered at $x$ with sides of length $\ell$ parallel to the axes.
	
	The second, and more natural, definition for $\BMO(Q_0)$ is one given by a bounded mean oscillation condition over subcubes $Q$ of $Q_0$. Such sets are a special case of the former, and so this second definition of $\BMO(Q_0)$ results in a space no smaller than the first. In fact, the opposite inclusion is true as well. For any $Q(x,\ell)\cap Q_0$ where $x\in Q_0$, it contains a subcube $Q_1$ of $Q_0$ of side length $\min\{\ell/2,\ell_0\}$ and is contained inside a subcube $Q_2$ of $Q_0$ of side length $\min\{\ell,\ell_0\}$. Here, $\ell_0$ is the side length of $Q_0$. It follows, see for instance \cite[Property O7]{BDG2} that for any locally integrable function on $Q_0$,
	$$
	\mathcal{O}(f, Q(x,\ell)\cap Q_0)\leq \frac{|Q_2|}{|Q_1|}\,\mathcal{O}(f, Q_2)\leq 2^n \,\mathcal{O}(f, Q_2),
	$$
	and so the two notions of $\BMO(Q_0)$ coincide with norms that are equivalent up to a dimension-dependent constant.
	
	The same arguments show that $\VMO(Q_0)$ as defined here coincides with the subspace of $\BMO(Q_0)$ of functions that having vanishing mean oscillation along subcubes of $Q_0$ with vanishing side length. As Lebesgue measure is Ahlfors $n$-regular and exhibits a $1$-annular decay property, we therefore have the following Euclidean corollary:
	
	\begin{corollary}\label{cor:Euclid}
		Let $0\leq \alpha<n$ and $f\in\BMO(Q_0)$. Then, $M^\alpha f \in \BLO(Q_0)$ with $\|M^\alpha f\|_{\BLO(Q_0)}\leq C\|f\|_{\BMO(Q_0)}$, where $C$ depends on $n$ and $\alpha$. Moreover, if $f\in\VMO(Q_0)$, then $M^\alpha f\in\VMO(Q_0)$.
	\end{corollary}
	
	
	
	\section{Continuity}\label{sec:cont}
	
	From \eqref{contract}, it is easy to see that if $M^\alpha$ is bounded on a function space $Y$ for which $|f|\leq |g|$ almost everywhere implies that $\|f\|_{Y}\leq\|g\|_{Y}$, then it is also continuous on $Y$. For instance, the $L^p$-norm and the $L^{p,\infty}$-quasinorm both satisfy such a property, and so it follows from Proposition \ref{prop:LocBound} that $M^\alpha$ is continuous from $L^1(X)$ to $L^{\frac{Q}{Q-\alpha},\infty}(X)$ for $0\leq\alpha<Q$ and from $L^p(X)$ to $L^{p^*}(X)$ for $1<p<\infty$ and $0\leq\alpha\leq Q/p$. 
	
	The same argument cannot be carried through for the $\BMO$-seminorm, however. In this section, we give examples that illustrate the lack of continuous action of maximal functions on $\BMO$ and $\VMO$.
	
	\begin{example}[Discontinuity of $M$ in the unbounded case]
		Let $X=(0,\infty)\subset\R$ with the Lebesgue measure and the Euclidean metric. In this example, we show that $M$ is not continuous on $\VMO(X)$, and thus neither on $\BMO(X)$.
		
		
		Define $f:X\to\R$ by 
		\[
		f(x)=
		\begin{cases}
			1-x,&0<x\le 1\\
			0,&1<x.
		\end{cases}
		\]
		For $n\in\N$, let $f_n=f-n.$ Then $f,f_n\in\VMO(X)$, and since $f_n$ and $f$ differ by a constant, we immediately have that $f_n\to f$ in $\BMO(X).$  However, it follows that 
		\[
		|f_n(x)|=
		\begin{cases}
			n-1+x,&0<x\le 1\\
			n,&1<x.
		\end{cases}
		\]
		By taking the supremum over arbitrarily large balls, it follows that $Mf_n\equiv n$ on $X.$  However, $Mf$ is not equal to a constant almost everywhere on $X,$ and so we have that $$\|Mf_n-Mf\|_{\BMO(X)}=\|Mf\|_{\BMO(X)}>0.$$ 
		Therefore, $Mf_n$ does not converge to $Mf$ in $\BMO(X)$, and so $M$ is not continuous on $\VMO(X).$
	\end{example}
	
	\begin{example}[Discontinuity of $M$ in the bounded case]
		Let $X=(0,2)\subset\R$, again with the Lebesgue measure and the Euclidean metric. Modifying the above example, we now show that $M$ is not continuous on $\VMO(X)$, and thus neither on $\BMO(X)$.
		
		Let $f:X\to\R$ be given by 
		\[
		f(x)=
		\begin{cases}
			0,&0< x\le1\\
			x-1,&1<x<2.
		\end{cases}
		\]
		For $n\in\N,$ let $f_n=f-n.$  As above, $f_n\to f$ in $\BMO(X)$ and $f,f_n\in\VMO(X).$  For $x\in(0,1)$, $1<\delta<2,$ and $0<\tau<x$, we have that 
		\[
		0<\fint_{(\tau,\delta)}|f|\le\fint_{(x,\delta)}|f|=\frac{(\delta-1)^2}{2(\delta-x)}.
		\]  
		Hence by continuity of the integral and the fact that $f=0$ on $(0,1)$, we have that, 
		\[
		Mf(x)=\sup_{1<\delta<2}\,\sup_{0<\tau<x}\fint_{(\tau,\delta)}|f|=\sup_{1<\delta<2}\fint_{(x,\delta)}|f|=\sup_{1<\delta<2}\frac{(\delta-1)
			^2}{2(\delta-x)}=\frac{1}{2(2-x)},
		\]
		and so $Mf$ is not constant almost everywhere on $(0,1).$
		
		Now, we have that 
		\[
		|f_n(x)|=
		\begin{cases}
			n,&0<x\le 1\\
			n+1-x&1<x<2,
		\end{cases}
		\]
		and so for $x\in(0,1),$ it follows that $Mf_n(x)=n.$  Hence,
		\begin{align*}
			\|Mf_n-Mf\|_{\BMO(X)}&\ge\cO(Mf_n-Mf,(0,1))\\
			&=\cO(n-Mf,(0,1))=\cO(Mf,(0,1))>0.
		\end{align*}
		Therefore, $Mf_n$ does not converge to $Mf$ in $\BMO(X),$ and so $M$ is not continuous on $\VMO(X).$
	\end{example}

	\begin{example}[Discontinuity of $M^\alpha$ in the bounded case]
		Let $X=(0,2)\subset\R$, again with the Lebesgue measure and the Euclidean metric. Continuing with the above example, we now show that $M^\alpha$ is not continuous on $\VMO(X)$, and thus neither on $\BMO(X)$.
		
		
		Let $f$ and $f_n$ be as in the previous example, and let $0<\alpha<1.$  Then as above, $f,f_n\in\VMO(X)$ and $f_n\to f$ in $\BMO(X).$  We first note that $M^\alpha f$ is not constant on the interval $(0,1).$  Indeed for $x\in(0,1)$, $1<\delta<2,$ and $0<\tau<x,$ we have that 
		\begin{align*}
			0<\left(\frac{\delta-\tau}{2}\right)^\alpha\fint_{(\tau,\delta)}|f|&=\frac{1}{2^\alpha(\delta-\tau)^{1-\alpha}}\int_{(\tau,\delta)}|f|.
		\end{align*}
		The right hand side of this expression is increasing in $\tau.$   
		As such, by continuity of the integral, and the fact that $f$ is zero on the interval $(0,1),$ it follows that,  
		\begin{align*}
			M^\alpha f(x)&=\sup_{1<\delta<2}\,\sup_{0<\tau<x}\left(\frac{\delta-x}{2}\right)^\alpha\fint_{(\tau,\delta)}|f|\\
			&=\sup_{1<\delta<2}\left(\frac{\delta-x}{2}\right)^\alpha\fint_{(x,\delta)}|f|=\sup_{1<\delta<2}\frac{(\delta-1)^2}{2^{\alpha+1}(\delta-x)^{1-\alpha}}=\frac{1}{2^{\alpha+1}(2-x)^{1-\alpha}},
		\end{align*}
		and so $M^\alpha f$ is not constant almost everywhere on $(0,1).$
		
		Now, for $x\in (0,1),$ $0<\tau<x$, and $x<\delta\le 1$, we have that 
		\begin{align*}
			\left(\frac{\delta-\tau}{2}\right)^\alpha\fint_{(\tau,\delta)}|f_n|=\left(\frac{\delta-\tau}{2}\right)^\alpha n\le\left(\frac{1-\tau}{2}\right)^\alpha n=\left(\frac{1-\tau}{2}\right)^\alpha\fint_{(\tau,1)}|f_n|.
		\end{align*}
		This follows since $|f_n|=n$ on $(0,1].$  Likewise for $1\le\delta<2$, we have that 
		\begin{align*}
			\left(\frac{\delta-\tau}{2}\right)^\alpha\fint_{(\tau,\delta)}|f_n|\le
			\left(\frac{\delta}{2}\right)^\alpha\fint_{(0,\delta)}|f_n|&=\frac{1}{2^\alpha\delta^{1-\alpha}}\left(n\delta-(\delta-1)^2/2\right).
		\end{align*}
		The inequality above holds since $\fint_{(\tau,\delta)}|f_n|\le\fint_{(0,\delta)}|f_n|.$  By combining these two cases, to compute $M^\alpha f_n(x)$, it suffices to take the supremum of the right hand side of the above expression over $1\le\delta<2.$  For $n>1/\alpha,$ the right hand side is increasing in $\delta,$ and so for such $n$, we have that 
		\[
		M^\alpha f_n(x)=\sup_{1\le\delta<2}\frac{1}{2^\alpha\delta^{1-\alpha}}\left(n\delta-(\delta-1)^2/2\right)=n-1/4.
		\]
		Hence for sufficiently large $n$, $M^\alpha f_n$ is constant on $(0,1)$, and so it follows that 
		\[
		\|M^\alpha f_n-M^\alpha f\|_{\BMO(X)}\ge\cO(M^\alpha f_n-M^\alpha f,(0,1))=\cO(M^\alpha f,(0,1))>0,
		\]
		where the last inequality holds since $M^\alpha f$ is not constant almost everywhere on $(0,1).$
		Therefore, $M^\alpha f_n$ does not converge to $M^\alpha f$ in $\BMO(X),$ and so $M^\alpha$ is not continuous on $\VMO(X).$
		
	\end{example}
	
	{\bf Acknowledgments.} The authors wish to thank Galia Dafni and Nageswari Shanmugalingam for valuable conversations related to this project and for their careful reading of the manuscript.

	\medskip
	
	\noindent Address:\\
	
	\vspace{-0.3cm}
	
	\noindent R.G.: Department of Mathematical Sciences, P.O. Box 210025, University of
	Cincinnati, Cincinnati, OH 45221--0025, U.S.A. \\
	\noindent E-mail: {\tt ryan.gibara@gmail.com}\\
	
	\vspace{-0.3cm}
	
	\noindent J.K.: Department of Mathematical Sciences, P.O. Box 210025, University of
	Cincinnati, Cincinnati, OH 45221--0025, U.S.A. \\
	\noindent E-mail: {\tt klinejp@mail.uc.edu}

\end{document}